\documentclass[12pt]{amsart}
\usepackage{amsmath,amssymb,amsbsy,amsfonts,latexsym,amsopn,amstext,
                                               amsxtra,euscript,amscd,bm}
\usepackage{comment}
\usepackage{graphicx}

\usepackage{mathtools}
\usepackage{units}
\usepackage{float}
\restylefloat{table}
\usepackage[english]{babel}                    
\usepackage{url}
\usepackage[colorlinks,linkcolor=blue,anchorcolor=blue,citecolor=blue]{hyperref}
\usepackage{color}
\usepackage[utf8]{inputenc}

\usepackage{mathtools}

\usepackage{listings}
\usepackage{xcolor}

\usepackage{blindtext}

\usepackage[multiple]{footmisc}
\usepackage{bigfoot}

\DeclareNewFootnote{AAffil}[arabic]
\definecolor{codegreen}{rgb}{0,0.6,0}
\definecolor{codegray}{rgb}{0.5,0.5,0.5}
\definecolor{codepurple}{rgb}{0.58,0,0.82}
\definecolor{backcolour}{rgb}{0.95,0.95,0.92}

\lstdefinestyle{mystyle}{
    backgroundcolor=\color{backcolour},   
    commentstyle=\color{codegreen},
    keywordstyle=\color{magenta},
    numberstyle=\tiny\color{codegray},
    stringstyle=\color{codepurple},
    basicstyle=\ttfamily\footnotesize,
    breakatwhitespace=false,         
    breaklines=true,                 
    captionpos=b,                    
    keepspaces=true,                 
    numbers=left,                    
    numbersep=5pt,                  
    showspaces=false,                
    showstringspaces=false,
    showtabs=false,                  
    tabsize=2
}

\begin{document}

\newtheorem*{note}{Note}
\newtheorem{problem}{Problem}
\newtheorem{theorem}{Theorem}
\newtheorem{lemma}[theorem]{Lemma}
\newtheorem{claim}[theorem]{Claim}
\newtheorem{corollary}[theorem]{Corollary}
\newtheorem{prop}[theorem]{Proposition}
\newtheorem{definition}{Definition}
\newtheorem{question}[theorem]{Question}
\newtheorem{conjecture}{Conjecture}
\def\cA{{\mathcal A}}
\def\cB{{\mathcal B}}
\def\cC{{\mathcal C}}
\def\cD{{\mathcal D}}
\def\cE{{\mathcal E}}
\def\cF{{\mathcal F}}
\def\cG{{\mathcal G}}
\def\cH{{\mathcal H}}
\def\cI{{\mathcal I}}
\def\cJ{{\mathcal J}}
\def\cK{{\mathcal K}}
\def\cL{{\mathcal L}}
\def\cM{{\mathcal M}}
\def\cN{{\mathcal N}}
\def\cO{{\mathcal O}}
\def\cP{{\mathcal P}}
\def\cQ{{\mathcal Q}}
\def\cR{{\mathcal R}}
\def\cS{{\mathcal S}}
\def\cT{{\mathcal T}}
\def\cU{{\mathcal U}}
\def\cV{{\mathcal V}}
\def\cW{{\mathcal W}}
\def\cX{{\mathcal X}}
\def\cY{{\mathcal Y}}
\def\cZ{{\mathcal Z}}

\def\A{{\mathbb A}}
\def\B{{\mathbb B}}
\def\C{{\mathbb C}}
\def\D{{\mathbb D}}
\def\E{{\mathbb E}}
\def\F{{\mathbb F}}
\def\G{{\mathbb G}}
\def\I{{\mathbb I}}
\def\J{{\mathbb J}}
\def\K{{\mathbb K}}
\def\L{{\mathbb L}}
\def\M{{\mathbb M}}
\def\N{{\mathbb N}}
\def\O{{\mathbb O}}
\def\P{{\mathbb P}}
\def\Q{{\mathbb Q}}
\def\R{{\mathbb R}}
\def\S{{\mathbb S}}
\def\T{{\mathbb T}}
\def\U{{\mathbb U}}
\def\V{{\mathbb V}}
\def\W{{\mathbb W}}
\def\X{{\mathbb X}}
\def\Y{{\mathbb Y}}
\def\Z{{\mathbb Z}}

\def\ep{{\mathbf{e}}_p}
\def\em{{\mathbf{e}}_m}
\def\eq{{\mathbf{e}}_q}

\def\scr{\scriptstyle}
\def\\{\cr}
\def\({\left(}
\def\){\right)}
\def\[{\left[}
\def\]{\right]}
\def\<{\langle}
\def\>{\rangle}
\def\fl#1{\left\lfloor#1\right\rfloor}
\def\rf#1{\left\lceil#1\right\rceil}
\def\le{\leqslant}
\def\ge{\geqslant}
\def\eps{\varepsilon}
\def\mand{\qquad\mbox{and}\qquad}

\def\sssum{\mathop{\sum\ \sum\ \sum}}
\def\ssum{\mathop{\sum\, \sum}}
\def\ssumw{\mathop{\sum\qquad \sum}}

\def\vec#1{\mathbf{#1}}
\def\inv#1{\overline{#1}}
\def\num#1{\mathrm{num}(#1)}
\def\dist{\mathrm{dist}}

\def\fA{{\mathfrak A}}
\def\fB{{\mathfrak B}}
\def\fC{{\mathfrak C}}
\def\fU{{\mathfrak U}}
\def\fV{{\mathfrak V}}

\newcommand{\bflambda}{{\boldsymbol{\lambda}}}
\newcommand{\bfxi}{{\boldsymbol{\xi}}}
\newcommand{\bfrho}{{\boldsymbol{\rho}}}
\newcommand{\bfnu}{{\boldsymbol{\nu}}}

\def\GL{\mathrm{GL}}
\def\SL{\mathrm{SL}}

\def\Hba{\overline{\cH}_{a,m}}
\def\Hta{\widetilde{\cH}_{a,m}}
\def\Hb1{\overline{\cH}_{m}}
\def\Ht1{\widetilde{\cH}_{m}}

\def\flp#1{{\left\langle#1\right\rangle}_p}
\def\flm#1{{\left\langle#1\right\rangle}_m}
\def\dmod#1#2{\left\|#1\right\|_{#2}}
\def\dmodq#1{\left\|#1\right\|_q}

\def\Zm{\Z/m\Z}

\def\Err{{\mathbf{E}}}

\newcommand{\comm}[1]{\marginpar{%
\vskip-\baselineskip 
\raggedright\footnotesize
\itshape\hrule\smallskip#1\par\smallskip\hrule}}

\def\xxx{\vskip5pt\hrule\vskip5pt}

\newenvironment{nouppercase}{%
  \let\uppercase\relax%
  \renewcommand{\uppercasenonmath}[1]{}}{}
  

\title{A note on medium and short character sums}
\author{Matteo Bordignon }
\address{KTH Royal Institute of Technology, Stockholm \newline and \newline Charles University, Faculty of Mathematics and Physics, Department of Algebra, Sokolovská 83, 186 00 Praha 8, Czech Republic Department of Mathematics}
\email{matteobordignon91@gmail.com}
\date{\today
}


\begin{abstract}
Following the work of Hildebrand we improve the P\'{o}lya-Vinogradov inequality in a specific range, we also give a general result that shows its dependency on Burgess bound and at last we improve the range of validity for a special case of Burgess' character sum estimate.
\end{abstract}
\begin{nouppercase}
\maketitle
\end{nouppercase}
It is of high interest studying the possible upper bounds of the following quantity
\begin{equation}
\label{eq:S}
S(N, \chi):=\left|\sum_{n=1}^N \chi(n) \right|,
\end{equation}
with $N\in \mathbb{N}$ and $\chi$ a non-principal Dirichlet character modulo $q$. The famous P\'{o}lya--Vinogradov inequality tells us that for any $\chi$ non-principal 
\begin{equation*}
S(N, \chi)\ll \sqrt{q}\log q,
\end{equation*}
and aside for the implied constant, this is the best known result. Paley in \cite{Paley} proved that for infinitely many characters we have
\begin{equation*}
\max_{N} \left|\sum_{n=1}^N\chi(n)\right|\gg \sqrt{q} \log \log q.
\end{equation*}
On the other hand Montgomery and Vaughan \cite{MV} showed, assuming the Generalized Riemann Hypothesis (GRH), we have that 
\begin{equation*}
\left|\sum_{n=1}^N\chi(n)\right|\ll \sqrt{q} \log \log q.
\end{equation*} 
The best known asymptotic constant in \eqref{eq:S} for primitive characters is $\frac{69}{70\pi 3 \sqrt{3}}+o(1)$, if $\chi$ is even from \cite{Granville} and $\frac{1}{ 3 \pi}+o(1)$ if $\chi$ is odd from \cite{Hild1}. For the best completely explicit constant see \cite{Bordignon}, \cite{B-K} and \cite{F-S}. For primitive characters of odd order $g$ Granville and Soundararajan improved the P\'{o}lya--Vinogradov inequality, proving the following bound
$$ \left|\sum_{n=1}^N\chi(n)\right|\ll \sqrt{q}(\log q)^{1-\delta_g/2+o(1)},$$
with $\delta=1-\frac{g}{\pi}\sin \frac{\pi}{g}$. Another interesting result is Theorem 1.1 in \cite{F} by E. Fouvry, E. Kowalski, P. Michel, C. S. Raju, J. Rivat, and K. Soundararajan, where they extend the so called P\'{o}lya--Vinogradov range.
It is interesting to note that it appears that $S(N, \chi)$ assumes its maximum for $N\approx q$, see the work by Bober et al. in \cite{Bober} and the one by Hildebrand, Corollary 3 of \cite{Hild1}, which proves that for even characters we have that $N=o(q)$ implies $S(N,\chi) =o(\sqrt{q}\log q)$. \newline
In this paper we will first give an improved version of the P\'{o}lya--Vinogradov inequality for a limited range, drawing inspiration from the work of Hildebrand in \cite{Hild_PV} and \cite{Hild1}. The interesting aspect of this result is that it does not use Burgess bound as it only relies on Montgomery and Vaughan \cite[Corollary 1]{MV} and that it allows to prove the best possible P\'{o}lya--Vinogradov inequality in a certain range. Specifically, we prove the following result that follows from Theorem \ref{theo:PVHild} taking $a(q)=(\log q)^{c}$, with $c>4+2\epsilon$, and $R(q)=(\log q)^{2+\epsilon}$ and going from primitive to non principal characters as done by Hildebrand in \cite{Hild_PV}.
\begin{theorem}
\label{theo:PVb}
Take $N$ and $q$ such that
\begin{equation*}
\frac{q}{(\log q)^{c+2}}<N<\frac{q}{2(\log q)^{2}},
\end{equation*}
for any fixed $c>0$.
Then, for $\chi$ an even non-principal character mod $q$, we have
\begin{align*}
\sum_{n=1}^N \chi(n) \ll_c \sqrt{q} \log \log q.
\end{align*}
\end{theorem}
Following the work of Hildebrand in \cite{Hild_PV}, we show that the best possible Burgess bound, see \cite{G-S3}, comes slightly short from proving the best possible P\'{o}lya--Vinogradov inequality, see \cite{MV}.
Specifically, we assume the following Burgess-like result, which is the one proven in \cite{G-S3}, assuming GRH, but with a stronger upper bound.
\begin{conjecture}
\label{Burgess}
For any non-principal character $\chi$ modulo $q$ and $x$ such that $\log x / \log \log q \rightarrow \infty$ we have
\begin{equation*}
\sum_{n\le x}\chi(n)\ll_{\epsilon} \frac{x}{(\log x)^{3+\epsilon}},
\end{equation*}
for a fixed $\epsilon>0$.
\end{conjecture}
This allows us to prove the following result, which comes $(\log \log q)^{\epsilon}$ short from the optimal result proven, assuming GRH, in \cite{MV}.
\begin{theorem}
\label{theo:GRHB}
Assuming Conjecture \ref{Burgess} and taking $\chi$ any non-principal character  modulo $q$, then we have
\begin{equation*}
\sum_{n\le x}\chi(n)\ll_{\epsilon} \sqrt{q}(\log \log q)^{1+\epsilon}.
\end{equation*}
\end{theorem}
We then focus on short character sums.
Let $\chi$ be a non-principal character modulo a prime $p$. Estimates of the type
\begin{equation}
\label{eq:B}
\left| \sum_{n\le N}\chi(n)\right|\le \epsilon N \quad (N \ge N_0(\epsilon,p)),
\end{equation}
are of great importance in number theory. By the P\'olya--Vinogradov inequality \eqref{eq:B} holds with $N_0(\epsilon,p)=(\log p)\sqrt{p}/\epsilon$, and Burgess' character sum estimate \cite{Burgess1, Burgess2} yields \eqref{eq:B} with $N_0(\epsilon,p)=p^{1/4+\delta}$ for any $\epsilon, \delta >0$ and $p\ge p_0(\epsilon,\delta)$. Hildebrand in \cite{Hild} proved  that given $\epsilon>0$ there exist $\delta=\delta(\epsilon)>0$ and $p_0(\epsilon)\ge 2$ such that for any $p\ge p_0(\epsilon)$ \eqref{eq:B} holds with $N_0(\epsilon,p)=p^{1/4-\delta}$, with $\delta(\epsilon)=\exp(-C(\epsilon^{-2}+1))$ with $C$ a sufficiently large constant. Using the result by Elliot in \cite{E} it is possible to take $\delta(\epsilon)=C\epsilon^{19}$ and using Granville and Soundararajan \cite[Corollary 3]{G-S1} $\delta(\epsilon)=C\epsilon^{2.76}$, that is currently the state of the art result. Also in \cite[$\S$ 9]{G-S2} Granville and Soundararajan show that \cite[Corollary 3]{G-S1} is optimal for general multiplicative functions, but it is worth to note that the  ``worst'' function they give is not a Dirichlet character. It is thus interesting to ask if it is possible to improve \cite[Corollary 3]{G-S1} for Dirichlet characters. In this paper we show that this can be done for real multiplicative functions proving Lemma \ref{lemma:diff}, we thus improve the size of the $\delta(\epsilon)$ in the chase of non-principal real characters. While this result follows easily from the current techniques, we were not able to find it in the literature and believe it is of interest. We obtain the following result that appears optimal with the current techniques.
\begin{theorem}
\label{theo:main}
Given $\epsilon>0$ there exist $\delta>0$ and $p_0(\epsilon)\ge 2$ such that for any non-principal real character $\chi$ modulo a prime $p\ge p_0(\epsilon)$ \eqref{eq:B} holds with $N_0(\epsilon,p)=p^{1/4-\delta}$ with
\begin{equation*}
\delta(\epsilon)=C\epsilon^{2},
\end{equation*}
with $C$ a sufficiently large absolute constant.
\end{theorem}
In Section \ref{sec:1} we prove a general version of Theorem \ref{theo:PVb}, Section \ref{sec:2} we prove Theorem \ref{theo:GRHB} and a more general version of it where we assume weaker versions of Conjecture \ref{Burgess} and in Section \ref{sec:4} we prove Theorem \ref{theo:main}.
\section*{Acknowledgements}
I would like to thank Bryce Kerr, P\"ar Kurlberg, Igor Shparlinski and Tim Trudgian for the useful comments and suggestions.\par
The research was partially supported by OP RDE project \newline No.
CZ.$02.2.69/0.0/0.0/18\_053/0016976$ International mobility of research,
technical and administrative staff at the Charles University.
\section{Improving P\'{o}lya--Vinogradov for a limited range without using Burgess bound}
\label{sec:1}
In this section we prove the following general version of Theorem \ref{theo:PVb}.
\begin{theorem}
\label{theo:PVHild}
Let $R:\mathbb{R}^+\rightarrow\mathbb{R}^+$ and $a:\mathbb{R}^+\rightarrow\mathbb{R}^+$ and $x$ such that $2 \le R < a(q)\le q$. Take $N$ and $q$ and assume
\begin{equation*}
\frac{qR(q)}{a(q)}<N<\frac{q}{R(q)}\left( 1-\frac{R(q)}{a(q)}\right).
\end{equation*}
Then, for $\chi$ an even primitive character, we have
\begin{align*}
\sum_{n=1}^N \chi(n) \le & \sqrt{q}2 \left( \frac{2}{\pi} \log a(q) +\frac{2}{\pi}\left( C +\log 2 +\frac{3}{a(q)}\right)\right)\\ & +O \left(\sqrt{q}\max\left(\log \frac{\log q}{\log a(q)}, \frac{(\log R(q))^{3/2}}{\sqrt{R(q)}}\log \frac{q}{a(q)}\right) \right).
\end{align*}
\end{theorem}
Let $f: \mathbb{Z} \rightarrow  \mathbb{C}$ be a multiplicative function with $| f(n)|\le 1$.
With $\alpha$ real and $e(\alpha)=\exp(2\pi i \alpha)$ write
\begin{equation*}
S(\alpha)=\sum^{N}_{n=1} f(n)e(n\alpha).
\end{equation*}
We will need \cite[Corollary 1]{MV}.
\begin{lemma}
\label{MV}
Suppose that $|\alpha - a/q| \le q^{-2}$, $(a,q)=1$ and $2 \le R\le q \le N/R$. Then
\begin{equation*}
\label{SA}
S(\alpha) \ll \frac{N}{\log N}+   \frac{N\log^{\frac{3}{2}} (R)}{\sqrt{R}}.
\end{equation*}
\end{lemma}
We now obtain a variation of \cite[Lemma 3]{Hild_PV}.
\begin{lemma}
\label{LI1}
Let $R:\mathbb{N}\rightarrow\mathbb{N}$ and $a:\mathbb{N}\rightarrow\mathbb{N}$ and $x$ such that $2 \le R \le a(q)$, $a(q)\le x \le q$ and $R/a<1$. Take $N$ and $q$ and assume
\begin{equation}
\label{eq:N}
\frac{qR(q)}{a(q)}<N<\frac{q}{R(q)}\left( 1-\frac{R(q)}{a(q)}\right).
\end{equation}
We have, uniformly for all primitive characters $\chi$  modulo $q$ as above,
\begin{equation*}
\left| \sum_{n \le x} \chi(n) e(\alpha n) \right| \ll \max\left(\frac{x}{\log x}, x\frac{(\log R(q))^{3/2}}{\sqrt{R(q)}}\right).
\end{equation*}
\end{lemma}
\begin{proof}
Set $M=[x]$, by Dirichlet's theorem there exist integers $r$ and $s$, where $(r, s)=1$ and $ 1\le s \le M/R$, such that
\begin{equation}
\label{alpha1}
\left| \frac{N}{q} - \frac{r}{s} \right| \le \frac{1}{sM/R}.
\end{equation}
Assuming $r=0$ we obtain, from \eqref{alpha1}, 
\begin{align*}
N\le \frac{qR}{sx},
\end{align*}
but this, together with the left-hand side of \eqref{eq:N}, gives 
\begin{align*}
\frac{1}{a(q)}<\frac{1}{sx},
\end{align*}
that is in contradiction with the assumption $a(q)\le x$ and $1\le s$. We can thus assume $r\neq 0$ which, \eqref{alpha1} together with $M=[x]$ and the assumptions $a(q)\le x \le q$ and $R/a<1$, gives 
\begin{align}
\label{eq:s}
s \ge \frac{q}{N}\left(r-\frac{R}{M}\right)\ge \frac{q}{N}\left(1-\frac{R(q)}{a(q)}\right)>0.
\end{align} 
Now from the right-hand side of \eqref{eq:N} we have
\begin{align*}
\frac{q}{N}>R(q)\left(1-\frac{R(q)}{a(q)}\right)^{-1}.
\end{align*}
Using this together with \eqref{eq:s} we obtain $s\gg R$.
Using $s\gg R$, the result follows from Lemma \ref{MV}. 
\end{proof}
We then need an explicit bound on a trigonometric sum, by Pomerance in \cite[Lemma 3]{Pomerance}. 
\begin{lemma}
\label{LI2}
Uniformly for $x \ge 1$ and with $\alpha$ real we have
\begin{equation*}
\sum_{n \le x}\frac{|\sin (\alpha n)|}{n} \le \frac{2}{\pi} \log x +\frac{2}{\pi}\left( C +\log 2 +\frac{3}{x}\right).
\end{equation*}
\end{lemma}
We can now prove Theorem \ref{theo:PVHild}.
\begin{proof}
We take $\chi$ primitive and start  with
\begin{equation*}
\chi(n)=\frac{1}{d (\overline{\chi})}\sum_{a=1}^q \overline{\chi}(a) e\left( \frac{an}{q} \right) = \frac{1}{d (\overline{\chi})} \sum_{0< |a|< q/2} \overline{\chi}(a) e\left( \frac{an}{q} \right),
\end{equation*}
where $d (\overline{\chi})$ is the Gaussian sum. Summing over $ 1 \le n \le N$, we obtain
\begin{align*}
\sum_{n=1}^N \chi(n)&= \frac{1}{d (\overline{\chi})} \sum_{0< |a|< q/2} \overline{\chi}(a) \sum_{n=1}^N e\left( \frac{an}{q} \right)\\ &= \frac{1}{d (\overline{\chi})} \sum_{0< |a|< q/2} \overline{\chi}(a) \frac{ e\left( \frac{aN}{q} \right)-1}{1- e\left( \frac{-a}{q} \right)}.
\end{align*}
It follows that
\begin{equation}
\label{eq:sum}
\sum_{n=1}^N \chi(n)\le \frac{\sqrt{q}}{2 \pi} \left| \sum_{0< |a|< q/2}\frac{\overline{\chi(a)}\left( e(\frac{aN}{q})-1\right)}{a}\right| + O(\sqrt{q}).
\end{equation}
Now we split the inner sum in two parts: $\Sigma_1$ with $0<|a|\le a(q)$ and $\Sigma_2$ with $a(q)<|a|<q/2$.\par
Now as $\chi$ is even, by partial summation and Lemma \ref{LI1}, we have
\begin{align*}
\Sigma_2 \ll  \max\left(\log \frac{\log q}{\log a(q)}, \frac{(\log R(q))^{3/2}}{\sqrt{R(q)}}\log \frac{q}{a(q)}\right).
\end{align*}
We now note that
\begin{equation*}
 \Sigma_1= 2i\sum\limits_{1\le a \le a(q)}\frac{\overline{\chi(a)}\sin(\frac{2\pi aN}{q})}{a},
\end{equation*}
and thus from Lemma \ref{LI2} we obtain
\begin{equation*}
  \left|\Sigma_1 \right| \le
     2 \left( \frac{2}{\pi} \log a(q) +\frac{2}{\pi}\left( C +\log 2 +\frac{3}{a(q)}\right)\right).
\end{equation*}
We thus obtain the desired result Theorem \ref{theo:PVHild}.
\end{proof}

\section{Improving P\'{o}lya--Vinogradov using a Burgess-like bound}
\label{sec:2}
In this section we aim to prove Theorem \ref{theo:GRHB}.
We first prove the following fundamental result.
\begin{lemma}
\label{PVconj}
Assuming Conjecture \ref{Burgess} holds, then we have the following result.
Take $a:\mathbb{R}^+\rightarrow\mathbb{R}^+$ such that $a(q)\le x \le q$ and with $x$ such that $2 \le (\log x)^{2+\epsilon} \le a(q)$. We have, uniformly for all primitive characters $\chi$  modulo $q$ as above and $\log a(q) / \log \log q \rightarrow \infty$,
\begin{equation*}
\left| \sum_{n \le x} \chi(n) e(\alpha n) \right| \ll_{\epsilon} \frac{x}{\log x}.
\end{equation*}
\end{lemma}
\begin{proof}
Set $N= \lfloor x\rfloor$, $R=(\log N)^{2+\epsilon}$. Taking $q$ big enough, we easily obtain $2 \le R \le N$. By Dirichlet's theorem there exist integers $r$ and $s$, where $(r, s)=1$ and $ 1\le s \le N/R$, such that
\begin{equation*}
\left| \alpha - \frac{r}{s} \right| \le \frac{1}{sN/R}.
\end{equation*}
If $ s\ge R$, the result follows from Lemma \ref{MV}, since
\begin{align*}
\frac{N}{\log N}+N\frac{(\log R)^{\frac{3}{2}}}{\sqrt{R}} \ll_{\epsilon} \frac{x}{\log x}.
\end{align*}
Now suppose $s < R$. By partial summation~follows
\begin{align*}
\left| \sum_{n \le x} \chi(n) e(\alpha n) \right| \ll  \left( 1+ \left| \alpha - \frac{r}{s}\right| x\right)\max_{u \le x} |T(u)|  \ll R \max_{u \le x} |T(u)|,
\end{align*}
where 
\begin{equation*}
T(u)=\sum_{n \le u} \chi(n)e\left( \frac{rn}{s} \right).
\end{equation*}
By grouping the terms of the sum $T(u)$ according to the value of $(n,s)$, we get
\begin{align*}
T(u)&=\sum_{dt=s} \sum_{\substack{dm \le u \\ (m,t)=1}} \chi(md) e\left( \frac{rm}{t}\right)\\&= \sum_{dt=s} \chi(d) \sum_{\substack{1\le a \le t\\ (a,t)=1}} e\left( \frac{ra}{t}\right)\sum_{\substack{m \le u/d\\ m=a \pmod t}} \chi(m)\\&
=\sum_{dt=s} \frac{\chi(d)}{\varphi (t)} \sum_{\psi \pmod t}\sum_{1\le a \le t}  e\left( \frac{ra}{t}\right)\overline{\psi}(a)\sum_{m \le u/d} \chi(m) \psi(m).
\end{align*}
Applying Conjecture \ref{Burgess} to the right-hand sum we obtain
\begin{align*}
 \sum_{n \le x} \chi(n)  e(\alpha n)  \ll_{\epsilon} R \frac{x}{(\log x)^{3+\epsilon}}\ll_{\epsilon} \frac{x}{\log x}.
\end{align*}
\end{proof}
We can now prove Theorem \ref{theo:GRHB}.
\begin{proof}
We again use \eqref{eq:sum} and split the inner sum in two parts: $\Sigma_1$ with $0<|a|\le a(q)$ and $\Sigma_2$ with $a(q)<|a|<q/2$.\par
Now using that $\chi$ is even, by partial summation and Lemma \ref{PVconj}, we have
\begin{align*}
\Sigma_2 \ll  \log \frac{\log q}{\log a(q)}.
\end{align*}
We then note that
\begin{equation*}
 \Sigma_1= 2i\sum\limits_{1\le a \le a(q)}\frac{\overline{\chi(a)}\sin(\frac{2\pi aN}{q})}{a},
\end{equation*}
and thus from Lemma \ref{LI2} we obtain
\begin{equation*}
  \left|\Sigma_1 \right| \le
     2 \left( \frac{2}{\pi} \log a(q) +\frac{2}{\pi}\left( C +\log 2 +\frac{3}{a(q)}\right)\right).
\end{equation*}
We thus obtain the desired result raking $\log a(q)=(\log \log q)^{1+\epsilon}$. 
\end{proof}
In general, we can assume the following Burgess-like result.
\begin{conjecture}
\label{Burgess1}
For any non-principal character $\chi$ modulo $q$ and $x$ such that $x \gg l(q)$, for a fixed $l:\mathbb{R}^+\rightarrow\mathbb{R}^+$, we have
\begin{equation*}
\sum_{n\le x}\chi(n)\ll \frac{x}{c(x)},
\end{equation*}
 for a certain $c:\mathbb{N}\rightarrow\mathbb{N}$.
\end{conjecture}
In the same way as in the proof of Lemma \ref{PVconj}, assuming Conjecture \ref{Burgess1}, we prove the following result.
\begin{lemma}
\label{PVconj1}
Assuming Conjecture \ref{Burgess1} holds for specific $l$ and $c$, then we have the following result.
Take $a:\mathbb{R}^+\rightarrow\mathbb{R}^+$ such that $a(q)\le x \le q$ and with $x$ such that $2 \le R(x) \le a(q)$. We have, uniformly for all primitive characters $\chi$  modulo $q$ as above and $x\gg l(q)$,
\begin{align*}
\left| \sum_{n \le x} \chi(n) e(\alpha n) \right| \ll x\max \left(\frac{1}{\log x}+\frac{(\log R (x))^{3/2}}{\sqrt{R(x)}},\frac{R(x)}{c(x)} \right)
\end{align*}
\end{lemma}
In the same way as in the proof of Theorem \ref{theo:GRHB}, using Lemma \ref{PVconj1}, we can prove the following more general result.
\begin{theorem}
Assuming Conjecture \ref{Burgess1} holds for specific $l$ and $c$, then we have the following result.
We have, uniformly for all even primitive characters $\chi$  modulo $q$ as above, $2 \le R(x) \le a(q)$ and $ l(q)\le a(q)$,
\begin{align*}
\left| \sum_{n \le x} \chi(n) \right| \ll  \log a(q) + \int_{a(q)}^q\frac{\max (x)}{x} dx,
\end{align*}
where
\begin{align*}
\max (x)=\max \left(\frac{1}{\log x},\frac{(\log R (x))^{3/2}}{\sqrt{R(x)}},\frac{R(x)}{c(x)} \right).
\end{align*}
\end{theorem}

\section{Hildebrand's version of Burgess bound}
\label{sec:4}
Theorem \ref{theo:main} is made possible by the following improvement of Lemma 4 in \cite{Hild_W}.
\begin{lemma}
\label{lemma:diff}
Let $f$ be a real multiplicative function satisfying $-1 \le f \le 1$, and let
\begin{equation*}
M(x)=M(x,f)=\frac{1}{x}\sum_{n\le x}f(x).
\end{equation*}
Then for $3\le x \le x'\le x^{\mathcal{O}(1)}$ we have
\begin{equation*}
\left| M(x')-M(x) \right|\ll \left( \frac{\log (2x'/x)}{\log x'}\right)^{1/2-\epsilon_1},
\end{equation*}
for any $\epsilon_1>0$ and where the implied constant is uniform. 
\end{lemma}
Before focusing on Lemma \ref{lemma:diff} we will use it to prove Theorem \ref{theo:main}.
\begin{proof} Applying Lemma \ref{lemma:diff} with $f=\chi$ (a non-principal real character modulo a prime $p$), $x=N\ge p^{1/4-\delta}$ and $x'=Np^{2\delta}(\ge p^{1/4+\delta})$, we obtain
\begin{equation*}
M(N,\chi)=M(Np^{2\delta},\chi)+\mathcal{O}\left(\delta^{1/2-\epsilon_1}\right).
\end{equation*}
The result now follows bounding $M(Np^{2\delta},\chi)$ using Burgess estimate.
\end{proof}
We are now left with proving Lemma \ref{lemma:diff}. The proof is based on two results. First we need the version of Hal\'{a}sz result given by Tenenbaum in \cite[pag. 343]{ten} which, defining 
\begin{equation*}
S(x,T):=\min_{|\gamma|\le 2T} \sum_{p\le x}\frac{1-\mathbf{R}\left(f(p)p^{-i\gamma}\right)}{p},
\end{equation*}
gives for any multiplicative function $|f(n)|\le 1$, $x\ge 3$ and $T\ge 1$
\begin{equation}
\label{eq:H}
M(x,f)\le (1+S(x,T))e^{-S(x,T)}+\frac{1}{\sqrt{T}}.
\end{equation}
This gives the first of the two results that we need.
\begin{theorem}
\label{theo:H-T}
For all real multiplicative functions $f$ with $|f|\le 1$, all $x\ge 3$ and with
\begin{equation*}
S'(x,t):=\min \left\{ \log \log x+O(1),  \sum_{p\le x}\frac{1-f(p)}{p}\right\}
\end{equation*}
we have
\begin{equation*}
S(x,f)\ll  (1+S'(x,t))\exp\left(- S'(x,t)\right).
\end{equation*}
\end{theorem}
\begin{proof}
The result follows from \eqref{eq:H}, observing that for $f(n)$ real we have
\begin{equation*}
S(x,T)=\min \left\{ \log \log x+O(1),  \sum_{p\le x}\frac{1-f(p)}{p}\right\}.
\end{equation*}
\end{proof}
The second is Proposition 4.1. in \cite{G-S}.
\begin{lemma}
\label{lemma:G_S}
For all real multiplicative functions $f$ with $-1 \le f\le 1$. Let $x$ be large, $1\le x\le x'$. Then
\begin{equation*}
\left| M(x')-M(x) \right|\ll \frac{\log 2 x'/x}{\log x'}\exp \left( \sum_{p\le x'}\frac{1-f(p)}{p} \right).
\end{equation*}
\end{lemma}
We can now prove Lemma \ref{lemma:diff}.
\begin{proof}
Given $x'\ge x \ge 3$, define $\delta$ by $x'=x^{1+\delta}$ and put
\begin{equation*}
R=\log \left( \frac{\log x'}{\log 2x'/x}\right)^{1/2+\epsilon}.
\end{equation*}
If we assume
\begin{equation*}
\sum_{p\le x'}\frac{1-f(p)}{p} \ge R,
\end{equation*}
we have that Lemma \ref{lemma:diff} follows from Theorem \ref{theo:H-T} and using that from $3 \le x \le x'\le x^{\mathcal{O}(1)}$ we have
\begin{equation*}
\sum_{x\le p\le x'}\frac{1-f(p)}{p}\ll 1.
\end{equation*}
On the other hand if we assume 
\begin{equation*}
\sum_{p\le x'}\frac{1-f(p)}{p} \le R,
\end{equation*}
we have that Lemma \ref{lemma:diff} follows from Lemma \ref{lemma:G_S}.
This concludes the proof.
\end{proof}

\end{document}